\documentclass{amsproc}

\usepackage[english]{babel}
\usepackage{hyperref}
\usepackage{amsmath,amsfonts,amssymb,amsthm,enumerate}
\usepackage{latexsym}
\usepackage{amscd}

\newcommand{\R}{\mathbb R}
\newcommand{\N}{\mathbb N}
\newcommand{\C}{\mathcal C}

\newcommand{\Q}{\mathbb Q}
\newcommand{\Z}{\mathbb Z}

\def\mod{\mathbf{Mod}}

\def\ent{\mathrm{ent}}

\def\End{\mathrm{End}}

\def\Im{\mathrm{Im}}

\global\def\mod#1{\mathrm{Mod}(#1)}

\def\ient{\widetilde\ent_v }
\def\rk{\mathrm{rk}}
\def\Fin{\mathrm{Fin}_v}
\def\I{\mathcal I}

\newtheorem{theorem}{Theorem}[section]
\newtheorem{lemma}[theorem]{Lemma}
\newtheorem{corollary}[theorem]{Corollary}
\newtheorem{proposition}[theorem]{Proposition}

\theoremstyle{definition}
\newtheorem{definition}[theorem]{Definition}

\theoremstyle{remark}

\numberwithin{equation}{section}

\begin{document}

\title{Intrinsic valuation entropy}

\author{Luigi Salce}
\address{Dipartimento di Matematica, Universit\`a di Padova,
Via Trieste 63 - 35121 Padova, ITALY}
\email{salce@math.unipd.it}
\thanks{}

\author{Simone Virili}
\address{Dept.\ de Matem\'{a}ticas, Universidad de Murcia,
30100 Espinardo, Murcia, SPAIN
}
\email{virili.simone@gmail.com}
\thanks{The second named author is supported by the Ministerio de Econom\'{\i}a y Competitividad of
Spain via a grant `Juan de la Cierva-Formaci\'on'. He is also supported by a research project
from the Ministerio de Econom\'{\i}a y Competitividad of Spain (MTM2014-53644-P) and from the
Fundaci\'on `S\'eneca' of Murcia (19880/GERM/15), both with a part of FEDER funds.}


\date{}

\begin{abstract}
We extend the notion of intrinsic entropy for endomorphisms of Abelian groups to endomorphisms of modules over an Archimedean non-discrete valuation domain $R$, using the natural non-discrete length function introduced by Northcott and Reufel for such a category of modules.
We prove that this notion of entropy is a length function for the category of $R[X]$-modules, it satisfies (a suitably adapted version of) the Intrinsic Algebraic Yuzvinski Formula and that it is essentially the unique invariant for $\mod {R[X]}$ with these properties.
\end{abstract}

\maketitle

\begin{center}
\dedicatory{Dedicated to Mike Prest on the occasion of his $65^{\text{th}}$ birthday}
\end{center}

\section{Introduction}

Several different notions of entropy have been introduced in the past few years in the algebraic setting, allowing  in-depth investigation of the dynamical behavior of endomorphisms of Abelian groups (see \cite{DGSZ,DG, DGSV, SZ}). Among these notions, the {\it intrinsic algebraic entropy}, introduced and investigated in \cite{DGSV} and
denoted by $\widetilde{\ent}$, is helpful for several reasons: it is meaningful for any kind of groups (not only for torsion groups); it satisfies the Addition Theorem on the whole category of Abelian groups; it is easily computable for linear transformations of finite dimensional vector spaces over the field of the rational numbers via the Intrinsic Algebraic Yuzvinski Formula (see \cite{DGSV,GV, SV1}).

The first notion of algebraic entropy for endomorphisms of discrete groups, introduced in \cite{AKM,W} and denoted here by $\ent$, has been deeply investigated in \cite{DGSZ}, and was then extended to modules over arbitrary rings via the tool of ``sub-additive invariants" (see \cite{SZ}), and later on widely explored in \cite{SVV} using the notion of  ``discrete length functions". On the other hand, the intrinsic algebraic entropy $\widetilde{\ent}$ did not receive a similar attention. Instead of trying to develop a general theory of the intrinsic algebraic entropy induced by a length function on the category of modules over an arbitrary ring, we pursue an intermediate goal, approaching this matter from a particular, but significant, point of view. The goal of this paper is to introduce and investigate the intrinsic algebraic entropy $\widetilde{\ent}_v$ over an archimedean non-discrete valuation domain $R$ with valuation $v\colon Q \to \Gamma \cup \{ \infty \}$, where $Q$ denotes the field of quotients of $R$ and $\Gamma$ is a dense subgroup of the totally ordered additive group $\R$ of the real numbers. In this setting there is a natural length function $L_v\colon\mod R \to \R_{\geq 0} \cup  \{ \infty \}$ induced by the valuation $v$, as proved by Northcott-Reufel in \cite{NR}. We will call the entropy $\widetilde{\ent}_v$ the {\it intrinsic valuation entropy}.

This setting has been already considered by Zanardo in \cite{Z}, who studied the valuation entropy $\ent_v$
induced by the length function $L_v$. Similarly to the Abelian group setting, this entropy is useful only for endomorphisms $\phi$ of torsion $R$-modules $M$. In particular, Zanardo computed $\ent_v(\phi)$ when $M$ is the cyclic $\phi$-trajectory generated by a torsion element (for these notions see the next section).

The particular point of view considered in this paper is significant for three reasons. Firstly, as proved by Northcott-Reufel in \cite{NR}, in this setting all the non-trivial non-discrete length functions over arbitrary valuation domains are multiples of the length functions $L_v$ mentioned above. Secondly, non-discrete valuation domains are a first nice example of non-Noetherian rings, for which some techniques used in the investigation of $\widetilde{\ent}$ and holding for Noetherian rings are not applicable. We need new techniques blending those used in \cite{DGSV} for Abelian groups and those used in \cite{SV} working with non-discrete length functions.
Finally, in \cite{SV} we were able to prove the Uniqueness Theorem for algebraic entropies induced by non-discrete length functions only in the particular setting of archimedean non-discrete valuation domains and of the length function $L_v$ above; this fact indicates that this setting is the appropriate first step {to be investigated} in the non-Noetherian case.

\medskip\noindent
{\bf Structure of the paper.} 
In Section 2 we give the definition of the intrinsic algebraic entropy $\widetilde{\ent}_v$ and prove its basic properties needed in the proof of the main results appearing in the next sections. We also prove that, similarly to what happens in the Abelian group setting, for endomorphisms $\phi$ of torsion $R$-modules the equality $\widetilde{\ent}_v(\phi)= \ent_v(\phi)$ holds.

\medskip
In Section 3 we prove that $\widetilde{\ent}_v$ is an upper-continuous invariant for the category $\mod{R[X]}$, looking as usual at an $R$-module $M$ with an endomorphism $\phi: M \to M$ as an $R[X[$-module via the action induced by $\phi$; we denote this module by $M_\phi$ and we set  $\widetilde{\ent}_v(M_\phi)= \widetilde{\ent}_v(\phi)$.
We also provide a formula characterizing $\widetilde{\ent}_v(M_\phi)$ as a supremum of lengths of suitable $R$-modules, a crucial tool in proving the Addition Theorem.

The long and articulate demonstration of this fundamental result takes all Section 4. It is splitted in two parts: first we prove the sub-additivity of $\widetilde{\ent}_v$, and then we prove with more efforts its super-additivity.

\medskip
In Section 5 we prove the Intrinsic Algebraic Yuzvinski Formula (IAYF, for short), in a suitable formulation adapted to the present situation. We will explain where this theorem comes from, comparing its statement with the analogous formula for the intrinsic algebraic entropy $\widetilde \ent$ for Abelian groups.

\medskip
In Section 6 we prove the Uniqueness Theorem, which states that the intrinsic valuation entropy is the unique length function for the category $\mod {R[X]}$ satisfying the IAYF and such that, composed with the tensor functor $ - \otimes _R R[X]$, equals the length function $L_v$.

\section{Definitions and preliminary facts}

Let $R$ be a ring and let $\C$ be a Serre subcategory of $\mod R$.  A function 
\[
L\colon\C\to \R_{\geq0}\cup\{\infty\}= \R^*,
\] 
is an {\em invariant} if $L(0)=0$ and $L(M)=L(M')$ whenever $M\cong M'$. Furthermore, 
\begin{enumerate}[\rm --]
\item $L$ is {\em additive} if
$L(M)=L(M')+L(M'')$ for any short exact sequence $0\to M'\to M\to M''\to 0$ in $\C$;
\item $L$ is {\em upper continuous } if $L(M)=\sup\{L(F):F\leq M,  \ F \text{ finintely generated}\}$.
\end{enumerate}
An invariant which is both additive and upper continuous is said to be a {\em length function}. If the finite values of a given length function $L$ form a discrete subset of $\R_{\geq0}$, then  $L$ is said to be {\it discrete}.

\subsection{The general setting}

In this paper we always assume that $R$ is a non-discrete archimedean valuation domain, that is, $R$ is $1$-dimensional and its value group $\Gamma(R)$ is a dense subgroup of the additive group of the real numbers $\R$, with valuation $v:Q \to \Gamma(R) \cup \{\infty \}$, where $Q$ is the field of fractions of $R$. We denote by $P$ the maximal ideal of $R$, which is not cyclic; every ideal in $R$ is either cyclic or countably generated. For a comprehensive treatment of modules over valuation domains we refer to \cite{FS} and \cite{FS1}.

\smallskip
Northcott and Reufel  \cite{NR} proved that, under our hypotheses on $R$, there is a unique (up to scalar multiplication) length function in $\mod R$ whose values do not form a discrete subset of $\R^*$. This function, that we denote here by
\[
L_v: \mod  R \to \R^* ,
\] 
is determined by the values on cyclic modules via the following formula: 
\[
L_v(R/I):= \inf_{r \in I} v(r).
\]
For an $R$-module $M$, the value $L_v(M)$ will be called the {\em valuation length} of $M$; if $L_v(M)< \infty$, we will sometime say that $M$ is $L_v$-finite.

\smallskip
Obviously, $L_v(R)=\infty$ and $L_v(R/P)=0$. More generally, any $R$-module which is not torsion has infinite valuation length, since it contains a submodule isomorphic to $R$, while it is immediate to check that $L_v(M)=0$ if and only if $M$ is semisimple. Furthermore, as noticed by Zanardo \cite{Z}, every finitely generated torsion $R$-module has finite valuation length. The following  lemmas give criteria for a torsion module to have finite valuation length and they will be quite important for us in the forthcoming sections.

\begin{lemma}\label{fin_gen_tor_fin}
Let $F$ be a finitely generated $R$-module. Then, any torsion submodule of $F$ has finite valuation length. 
\end{lemma}
\begin{proof}
It is well known (see \cite[V.2.9]{FS1}) that the torsion part $t(F)$ of $F$ is a summand of $F$, so it is finitely generated. By the above discussion, $L_v(t(F))<\infty$ and, given a torsion submodule $T$ of $ F$, the inclusion $T\leq t(F)$ gives that $L_v(T)\leq L_v(t(F))<\infty$.
\end{proof}

\begin{lemma}\label{bounded_is_finite}
Let $M=Q^n$ for some positive integer $n$, and let $R^n\leq K\leq N\leq Q^n$ be submodules such that $N/K$ is a bounded $R$-module. Then, $L_v(N/K)<\infty$.
\end{lemma}
\begin{proof}
By \cite[XII.1.1]{FS1}, $N/K$ is weakly polyserial, that is, it has a chain of submodules $0=H_0<H_1<\cdots<H_t=N/K$ such that $H_i/H_{i-1}$ is uniserial for all $i\leq t$. As $N/K$ is bounded, every factor $H_i/H_{i-1}$ is such. Since uniserial modules over an archimedean valuation domain are either cyclic or countably generated (see \cite[VII.1]{FS}), every factor $H_i/H_{i-1}$ is standard uniserial, that is, isomorphic to a module of the form $J/I$, for $0<I<J \leq R$ (see \cite[VII.1.3]{FS}), so it has finite valuation length. Thus, $L_v(N/K)=\sum_{i=1}^tL_v(H_i/H_{i-1})<\infty$, as desired. 
\end{proof}

Let us also remark that there are plenty of non-finitely generated modules of finite valuation length; for example, if $\Gamma(R)$ contains $\Z[1/2]$, the subring of $\Q$ generated by $1/2$, and $M=\bigoplus _{n \geq 1}R/I_n$, with $L_v(R/I_n)=2^{-n}$, then 
\[
L_v(M)=\sum_{n \geq 1} L_v(R/I_n)=\sum_{n \geq 1} 2^{-n}=1.
\] 

Let us conclude this subsection recalling that there are just, up to scalar multiplication, only two more non-trivial length functions of $\mod R$ (where by ``non-trivial'' we mean functions that assume some finite non-zero value): the {\em composition length} 
\[
\ell\colon \mod R\to \R^*,
\] 
characterized by the fact that $\ell(R/P)=1$, and the {\em torsion-free rank} 
\[
\rk_R\colon \mod R\to \R^*,
\]
characterized by the fact that $\rk_R(R)=1$. Notice that the set of finite values of both these functions is $\N$, which is clearly a discrete subset of $\R$.

\subsection{Trajectories, anti-trajectories, and inert submodules}\label{inert_subs}

Let us start fixing some conventions for $R[X]$-modules. Indeed, we use the notation $M_\phi$, with $M$ an $R$-module and $\phi\in\End_R(M)$, to denote the $R[X]$-module $M_{R[X]}$, where $X$ acts on $M$ as $\phi$. An $R$-homomorphism $\alpha\colon M\to N$ induces an $R[X]$-homomorphism $M_\phi\to N_\psi$ if and only if $\alpha \cdot \phi = \psi \cdot \alpha$. For more details on these notions we refer to \cite{SVV}.

Given an $R[X]$-module $M_\phi$, in this section we introduce and study a class of $R$-submodules of $M$ called $\phi$-inert. This family of submodules is meant to substitute the family of $L_v$-finite submodules of $M$ in the definition of the valuation entropy $\ent_v$. For this reason we will proceed as follows: we first recall the definition of $\ent_v$ and the role played by the $L_v$-finite submodules and their trajectories; in the second part of the subsection we introduce the (valuation) inert submodules, setting the bases for defining the intrinsic valuation entropy $\ient$, and we use the notion of anti-trajectory to produce a new, better behaved, inert submodule from a given one.

\medskip
Fix an $R[X]$-module $M_\phi$. For any $R$-submodule $K$ of $M$ and any integer $n \geq 1$ we define  the {\em partial $n$-th trajectory} of $K$ as follows
\[
T_n(\phi, K) = K + \phi K + \phi^2 K +\ldots + \phi^{n-1}K.
\]
Similarly, the (full) {\em trajectory} of $K$ is the following submodule of $M$:
\[
T(\phi,K) = \bigcup_{n>0} T_n(\phi,K)=\sum_{n\geq 0}\phi^nK.
\]
Notice that $T(\phi,K)$ is the $R[X]$-submodule of $M_\phi$ generated by $K$. The trajectory of a cyclic submodule $xR\leq M$ is simply denoted by $T(\phi,x)$.

\smallskip
Following the general treatment of \cite{SZ} and \cite{SVV}, the ({\em algebraic}) {\em valuation entropy} $\ent_v(M_\phi)$  was defined in \cite{Z} as follows. Let 
\[
\Fin(M)=\{K\leq M:L_v(K)< \infty\}.
\] 
Then, given $K \in \Fin(M)$, set
\[
\ent_v(\phi,K) = \lim_{n \to \infty} \frac{L_v(T_n(\phi,K ))}{n};
\]
this limit exists finite by the well-known Fekete's Lemma. In fact, one can prove as in \cite[Prop.\,3.2]{SV}  that 
\begin{equation}\label{suggesting_inerts}
\ent_v(\phi,K)=\inf_n L_v\left(\frac{T_{n+1}(\phi, K)}{T_n(\phi, K)}\right).
\end{equation}
One finally defines the {\em valuation entropy} of $\phi$ as 
\[
\ent_v(M_\phi)= \sup_{K \in \Fin(M)} \ent_v(\phi,K).
\]
We will commonly use also the notation $\ent_v(\phi)$ to denote $\ent_v(M_\phi)$.

\smallskip
The formula \eqref{suggesting_inerts}, which has its analog for the classical algebraic entropy $\ent$ in the setting of Abelian groups, suggests that $\ent_v(\phi,K)$ can be defined not just for $L_v$-finite submodules, but for any submodule $K$ such that $L_v(T_{n+1}(\phi, K)/T_n(\phi, K))$ is finite for all $n\geq 1$. As it turns out, this happens precisely when $L_v((K+\phi K)/K)$ is finite. This motivates the following definition.

\begin{definition}
In the above setting, an $R$-submodule $K$ of $M$ is  {\em (valuation) $\phi$-inert} provided 
$L_v((K+\phi K)/K)<\infty$.
\end{definition}

The family of all the $\phi$-inert submodules of $M$ is denoted by $\mathcal I_\phi (M)$.
Proceeding as in \cite[Lem.\,2.1]{DGSV}, one can verify that, if $H$ is $\phi$-inert in $M$, then $L_v(T_n(\phi, H )/H) < \infty$ for all $n\geq 1$ and, consequently, $T(\phi, H )/H$ is a torsion module. This fact will play a crucial role in the proof of Proposition \ref{F,F'}, a main step in proving the super-additivity of the intrinsic valuation entropy.

\begin{lemma} \label{restriction}
Let $\phi: M \to M$ be an  endomorphism of an $R$-module $M$, $H$ a $\phi$-invariant submodule and $K$ a $\phi$-inert submodule of $M$. Then $K \cap H$ is $\phi_{ \restriction H}$-inert in $H$.
\end{lemma}
\begin{proof}
Consider the following isomorphisms
$$\frac{(K \cap H) + \phi (K\cap H)}{K \cap H} \cong \frac{\phi (K\cap H)}{\phi (K\cap H) \cap K \cap H} \cong \frac{K+ \phi (K\cap H)}{K}.$$
The last term is a submodule of $(K+\phi K)/K$, therefore $L_v\left(\frac{(K \cap H) + \phi (K\cap H)}{K \cap H}\right) < \infty$.
\end{proof}

The next technical lemma will be useful in the proofs of Popositions \ref{freeinfr} and \ref{entropy_fg_coro_wrt_tors}.

\begin{lemma} \label{Kinert}
Let $M_\phi$ be an $R[X]$-module such that $\rk_R(M)=k<\infty$. Then  any $R$-submodule $K\leq M$ such that $K/tK\cong R^k$ and $L_v(tK)<\infty$ is $\phi$-inert. In particular, if $M$ is torsion-free and $F$ is a free submodule of maximal rank, then $F$ is $\phi$-inert.
\end{lemma}
\begin{proof}
Since $K/tK\cong R^k$, we get $K=tK\oplus F$, where $F \cong R^k$.
Consider the following short exact sequence
\[
0\to\frac{t(K+\phi K)}{tK}\to \frac{K+\phi K}{K}\to \frac{(K+\phi K)+tM}{K+tM}\to 0 \]
where the module on the right-hand side is $L_v$-finite, since it is torsion ($K+\phi K+tM$ and $K+tM$ have the same rank) and finitely generated (it is isomorphic to a quotient of $R^k$). On the other hand, to prove that the module on the left-hand side is $L_v$-finite, it is enough to show that $L_v(t(K+\phi K))<\infty$, since by hypothesis $L_v(tK)<\infty$.
But $t(K+\phi K)=tK +\phi(tK)+t(F+\phi F)$. The first two summands are $L_v$-finite, while the third summand, as the torsion part of a finitely generated module, is also $L_v$-finite by Lemma \ref{fin_gen_tor_fin}. Then the conclusion follows.
\end{proof}

We  are now  going to introduce a construction that takes as input a given inert submodule $K$ and produces as output a bigger inert submodule $A(\phi,K)$ with better properties. This procedure is dual to a construction introduced by Willis in \cite{Willis_endo}, in studying the scale function of continuous endomorphisms of totally disconnected locally compact groups. 
Indeed, for an $R[X]$-module $M_{\phi}$ and an $R$-submodule $K\leq M$, we define, by induction on $n\geq 1$, the {\em (partial) $n$-th anti-trajectory of $K$} as follows:
\begin{enumerate}[\rm --]
\item $A_1(\phi,K)=K$;
\item $A_{n+1}(\phi,K)=K+\phi^{-1} A_n(\phi,K)$.
\end{enumerate}
The {\em (full) anti-trajectory of $K$} is then defined by
\[
A(\phi,K)=\bigcup_{n}A_n(\phi,K).
\]
In the following lemma we show that there is a close relationship between trajectories and anti-trajectories of a given submodule. 

\begin{lemma}\label{lem_K_n}
In the above notation, the following statements hold true:
\begin{enumerate}[\rm (1)]
\item $\phi^{n}(\phi^{-1}A_n(\phi,K))\cap \phi^{n}K= T_{n}(\phi,K)\cap \phi^{n}K$;
\item $T_{n+1}(\phi,K)/T_{n}(\phi,K)\overset{}{\cong} {A_{n+1}(\phi,K)}/{\phi^{-1}A_{n}(\phi,K)}$, for all $n\geq 1$.
\end{enumerate}
\end{lemma}
\begin{proof}
(1) Let us start proving that, for all $n\geq 1$, $\phi^{n}(\phi^{-1}A_{n}(\phi,K))\subseteq T_{n}(\phi,K)$. We proceed by induction on $n$. For $n=1$ we get $\phi(\phi^{-1}K)=K\cap \Im(\phi)\subseteq K$. For $n> 1$,
\begin{align*}
\phi^{n}(\phi^{-1}A_{n}(\phi,K))&=\phi^{n-1}(A_{n}(\phi,K))\cap \Im(\phi)\subseteq \phi^{n-1} A_{n}(\phi,K)\\
&=\phi^{n-1}K+\phi^{n-1}(\phi^{-1}A_{n-1}(\phi,K))\\
&\subseteq \phi^{n-1}K+T_{n-1}(\phi,K)=T_{n}(\phi,K).
\end{align*}
On the other hand, a generic element $x\in T_{n}(\phi,K)\cap \phi^{n}K$ can be written as $x=\phi^{n}k_{n}=k_0+\phi k_1+\ldots+\phi^{n-1}k_{n-1}$, for $k_0,k_1,\ldots,k_{n}\in K$. Hence, the equality $\phi(\phi^{n-1}k_{n}-\phi^{n-2}k_{n-1}-\ldots-k_1)=k_0$ implies that
\[
\phi^{n-1}k_{n}-\phi^{n-2}k_{n-1}-\ldots-k_1\in \phi^{-1}K.
\]
Furthermore, the fact that $\phi(\phi^{n-2}k_{n}-\phi^{n-3}k_{n-1}-\ldots-k_2)\in k_1+\phi^{-1}K\subseteq K+\phi^{-1}K$, implies that
\[
\phi^{n-2}k_{n}-\phi^{n-3}k_{n-1}-\ldots-k_2\in \phi^{-1}(K+\phi^{-1}K)
\]
Going on this way, one proves that $k_{n}\in \phi^{-1}A_{n}(\phi,K)$, and so $x=\phi^{n}k_{n}\in \phi^{n}K\cap \phi^{n}(\phi^{-1}A_{n}(\phi,K))$. This show that $\phi^{n}(\phi^{-1}A_n(\phi,K))\cap \phi^{n}K\supseteq T_{n}(\phi,K)\cap \phi^{n}K$, so we are done.

\smallskip\noindent
(2) This is a consequence of part (1); in fact,
\begin{align*}
\frac{T_{n+1}(\phi,K)}{T_{n}(\phi,K)}&\cong \frac{\phi^{n}K}{T_{n}(\phi,K)\cap \phi^{n}K}\\
&= \frac{\phi^{n}K}{\phi^{n}(\phi^{-1}A_{n}(\phi,K))\cap \phi^{n}K}\\
&\cong \frac{\phi^{n}(K+\phi^{-1}A_{n}(\phi,K))}{\phi^{n}(\phi^{-1}A_{n}(\phi,K))}\\
&\cong \frac{\phi^{n}(A_{n+1}(\phi,K))}{\phi^{n}(\phi^{-1}A_{n}(\phi,K))} \ .
\end{align*}
To conclude we should just prove that the following map, induced by $\phi^n$:
\[
f\colon\frac{A_{n+1}(\phi,K)}{\phi^{-1}A_{n}(\phi,K)}\to \frac{\phi^{n}(A_{n+1}(\phi,K))}{\phi^{n}(\phi^{-1}A_{n}(\phi,K))}
\]
is an isomorphism. In fact, it is clearly surjective, while its injectivity follows from the fact that $\ker(\phi^{n})\subseteq \phi^{-n}K\subseteq \phi^{-1}A_{n}(\phi,K)$.
\end{proof}

We describe in the next proposition the trajectories of $A(\phi,K)$, showing that $A(\phi,K)$ is $\phi$-inert provided $K$ has the same property. 

\begin{proposition}\label{prop_K}
In the above notation, the following statements hold true:
\begin{enumerate}[\rm (1)]
\item $\phi^{-1}A(\phi,K)\subseteq A(\phi,K)$, so $A(\phi,K)\cap \phi A(\phi,K)=A(\phi,K)\cap \Im(\phi)$ and, more generally, for every $n \geq 1$, $\phi^{n-1}A(\phi,K)\cap \phi^n A(\phi,K)=\phi^{n-1}A(\phi,K)\cap \Im(\phi^n)$;
\item $T_n(\phi,A(\phi,K))\cap \phi^n A(\phi,K)=\phi^{n-1}(\phi A(\phi,K)\cap A(\phi,K))$, for all $n\geq 1$;
\item $T_{n+1}(\phi,A(\phi,K))/T_{n}(\phi,A(\phi,K))\cong A(\phi,K)/\phi^{-1}A(\phi,K)$, for all $n\geq 1$;
\item if $K$ is $\phi$-inert, then so is $A(\phi,K)$.
\end{enumerate}
\end{proposition}
\begin{proof}
\smallskip
(1) By definition, $\phi^{-1}A_n(\phi,K)\subseteq A_{n+1}(\phi,K)\subseteq A(\phi,K)$, for all $n\geq 1$. Thus, 
\[
\phi^{-1}A(\phi,K)=\phi^{-1}\left(\bigcup_{n\geq 1}A_n(\phi,K)\right)=\bigcup_{n\geq 1}\phi^{-1}A_n(\phi,K)\subseteq A(\phi, K).
\]
The inclusion $A(\phi,K)\cap \phi A(\phi,K)\subseteq A(\phi,K)\cap \Im(\phi)$ is obvious; the converse  follows from the inclusion  $\phi(\phi^{-1}A(\phi,K))\subseteq \phi(A(\phi,K))$ and $\phi(\phi^{-1}A(\phi,K))=A(\phi,K) \cap \Im(\phi)$. A similar argument proves that $\phi^{n-1}A(\phi,K)\cap \phi^n A(\phi,K)=\phi^{n-1}A(\phi,K)\cap \Im(\phi^n)$.

\smallskip\noindent
(2) The inclusion $\supseteq$ is clear. For the converse inclusion, let $x=\phi^{n-1}k_{n-1}+\ldots+\phi k_1+k_0=\phi^n k_n\in T_n(\phi,A(\phi,K))\cap \phi^n A(\phi,K)$ ($k_i \in A(\phi,K)$), and let us show that $x\in \phi^{n-1}(\phi A(\phi,K)\cap A(\phi,K))$. Since $\ker(\phi^{n-1})\subseteq \phi^{-n+1}A(\phi,K)\subseteq A(\phi,K)$, we get
\[
\phi^{n-1}(\phi A(\phi,K)\cap A(\phi,K))=\phi^{n}A(\phi,K)\cap \phi^{n-1}A(\phi,K).
\]
Hence, it is enough to show that $x\in \phi^{n-1}A(\phi,K)$. Indeed, the equality
\[
k_0=\phi^nk_n-(\phi^{n-1}k_{n-1}+\ldots+\phi k_1)
\]
shows that $k_0\in A(\phi,K)\cap \Im(\phi)=A(\phi,K)\cap \phi A(\phi,K)$. Hence, there is $k'_{1}\in A(\phi,K)$ such that $\phi k'_{1}=\phi k_1+ k_0$. Hence,
\[
\phi k'_{1}=\phi^nk_n-(\phi^{n-1}k_{n-1}+\ldots+\phi^2 k_2)
\]
that is, $\phi(k'_{1})\in \phi A(\phi,K)\cap \Im(\phi^2)=\phi A(\phi,K)\cap \phi^2(A(\phi,K))$. Hence, there is $k_2'\in A(\phi,K)$, such that $\phi^2 k_2'=\phi^2 k_2+ \phi k'_{1}=\phi^2k_2+\phi k_1+k_0$. Going on this way, we obtain an element $k'_{n-1}\in A(\phi,K)$ such that
\[
\phi^{n-1}k'_{n-1}=\phi^{n-1}k_{n-1}+\ldots+\phi k_1+k_0=\phi^nk_n
\]
concluding the proof of our claim.

\smallskip\noindent
(3) Consider the following sequence of isomorphisms
\begin{align*}
\frac{T_{n+1}(\phi,A(\phi,K))}{T_{n}(\phi,A(\phi,K))}&\cong \frac{\phi^n A(\phi,K)}{\phi^n A(\phi,K)\cap T_n(\phi,A(\phi,K))}\\
&=\frac{\phi^n A(\phi,K)}{\phi^{n-1}(\phi A(\phi,K)\cap A(\phi,K))}\\
&\cong \frac{\phi^{n-1} (\phi A(\phi,K)+ A(\phi,K))}{\phi^{n-1} A(\phi,K)} \\
&\cong \frac{\phi A(\phi,K)+ A(\phi,K)}{ A(\phi,K)} \\
\end{align*}
where the last isomorphism is induced by $\phi^{n-1}$, using that $\ker(\phi^{n-1})\subseteq A(\phi,K)$. Finally, by part (1), 
\[
\frac{\phi A(\phi,K)+ A(\phi,K)}{ A(\phi,K)} \cong \frac{\phi A(\phi,K)}{A(\phi,K)\cap \phi A(\phi,K)}=\frac{\phi A(\phi,K)}{A(\phi,K)\cap \Im(\phi)}\cong \frac{A(\phi,K)}{\phi^{-1}A(\phi,K)}
\]
where the last isomorphism is induced by $\phi$.

\smallskip\noindent
(4) Notice that $K+\phi^{-1}A(\phi,K)= A(\phi,K)$. In fact, $K+\phi^{-1}A(\phi,K)\subseteq A(\phi,K)$ by the first part, while for all $n\in\N$, $A_{n+1}(\phi,K)=K+\phi^{-1}A_n(\phi,K)\subseteq K+\phi^{-1}A(\phi,K)$. Applying $\phi$ to the equality $K+\phi^{-1}A(\phi,K)= A(\phi,K)$, we obtain that $\phi K+ (A(\phi,K)\cap \Im(\phi))=\phi A(\phi,K)$; in particular, $\phi A(\phi,K)+A(\phi,K)=\phi K+A(\phi,K)$. Therefore 
\[
\frac{A(\phi,K)+\phi A(\phi,K)}{A(\phi,K)}=\frac{\phi K+A(\phi,K)}{A(\phi,K)}\cong \frac{\phi K}{\phi K\cap A(\phi,K)}
\]
is an epic image of 
\[
\frac{\phi K}{\phi K\cap K}\cong \frac{K+\phi K}{K}
\]
showing that $(A(\phi,K)+\phi A(\phi,K))/A(\phi,K)$ has finite valuation length, that is, $A(\phi,K)$ is $\phi$-inert.
\end{proof}

\subsection{The intrinsic valuation entropy}

Let us begin this subsection defining the main object of study for this paper.

\begin{definition}
Consider an $R[X]$-module $M_\phi$, and let $H \in \mathcal I_\phi (M)$. The {\rm intrinsic valuation entropy of $\phi$ with respect to $H$} is 
\[
\ient(\phi,H)=\lim_{n \to \infty}  \frac{L_v(T_n(\phi, H )/H)}{n}.
\]
The  {\rm intrinsic valuation entropy of $M_\phi$} is then defined as
\[
\ient(M_\phi)=\sup_{H \in \mathcal I_\phi (M)} \ient(\phi,H).
\]
We will commonly use also the notation $\ient(\phi)$ to denote $\ient(M_\phi)$.
\end{definition}

Mixing the proofs of  \cite[Lem.\,3.2]{DGSV} and \cite[Prop.\,3.2]{SV}, we obtain the following

\begin{proposition}\label{inf}
Given an $R[X]$-module $M_\phi$ and $H \in \mathcal I_\phi (M)$, 
\[
\ient(\phi,H)=\inf_n L_v\left(\frac{T_{n+1}(\phi, H )}{T_n(\phi, H )}\right).
\]
\end{proposition}
\begin{proof}
In order to simplify notation set $T_n(\phi, H )=T_n$ for all $n \geq 1$. As proved in \cite{SZ}, $T_{n+1}/T_n$ is a quotient of $T_n/T_{n-1}$ for all $n>1$, hence 
\[L_v((T_{n+1}/H)/(T_n/H)) \leq L_v((T_n/H)/(T_{n-1}/H)).\]
Let $\alpha =\inf_n  L_v((T_n/H)/(T_{n-1}/H))$ and fix (arbitrarily) $\epsilon>0$. Then there exists an index $n_0$ such that
\begin{equation*}\label{eq1}
L_v((T_n/H)/(T_{n-1}/H)) < \alpha + \epsilon
\end{equation*}
for all $n \geq n_0$. Using the additivity of $L_v$, one can prove by induction on $k$ that
$$L_v(T_{n_0+k}/H)=L_v(T_{n_0}/H)+\sum_{i \leq 0 \leq k-1}L_v((T_{n_0+i+1}/H)/(T_{n_0+i}/H))$$
which implies
\begin{equation}\label{eq2}
L_v(T_{n_0}/H)+k \alpha \leq L_v(T_{n_0+k}/H) \leq L_v(T_{n_0}/H)+k (\alpha + \epsilon).
\end{equation}
Using the second inequality in \eqref{eq2} we obtain:
\[
\ient(\phi,H)=\lim_{k \to \infty}\frac{L_v(T_{n_0+k}/H)}{n_0+k} \leq \lim_{k \to \infty}\frac{L_v(T_{n_0}/H)+k(\alpha + \epsilon)}{n_0+k} = \alpha + \epsilon.
\]
Being $\epsilon >0$ arbitrary, we deduce that $\widetilde{\ent}_v(\phi,H) \leq \alpha$. For the converse inequality, using the first inequality in \eqref{eq2} we obtain:
\begin{equation*}
\ient(\phi,H)=\lim_{k \to \infty}\frac{L_v(T_{n_0+k}/H)}{n_0+k}  \geq \lim_{k \to \infty}\frac{L_v(T_{n_0}/H)+ k \alpha}{n_0+k} = \alpha.\qedhere
\end{equation*} 
\end{proof}

\begin{corollary}\label{inf_coro}
Given an $R[X]$-module $M_\phi$ and $H \in \mathcal I_\phi (M)$, the following hold:
\begin{enumerate}[\rm (1)]
\item $\ient(\phi,H)=\ient(\phi,T_n(\phi,H))$ for all $n \geq 1$;
\item if $H\in \Fin(M)\subseteq \I_\phi(M)$, then $\ient(\phi,H)=\ent_v(\phi,H)$.
\end{enumerate}
\end{corollary}
\begin{proof}
Both statements are direct consequences of Proposition \ref{inf}; for  (2) use also \eqref{suggesting_inerts}.
\end{proof}

The intrinsic valuation entropy $\widetilde{\ent}_v$ satisfies the following two typical properties of  algebraic entropies; since their proofs are mostly straightforward and very close to those of the analogous properties for other entropies, we leave them almost completely as exercises.

\begin{proposition}\label{conjugate}
Given an isomorphism of $R[X]$-modules $\alpha\colon M_\phi\to N_{\psi}$, so that $\psi=\alpha \phi \alpha^{-1}$, then $\ient(\phi)=\ient(\psi)$.
\end{proposition}
As usual, $\widetilde{\ent}_v$ can be viewed as a function $\mod{R[X]} \to \R^*$, mapping an $R[X]$-module $M_\phi$ to $\widetilde{\ent}_v(M_\phi)$. In this setting, Proposition \ref{conjugate}, together with the obvious fact that $\widetilde{\ent}_v(0)=0$, tells us that $\widetilde{\ent}_v$ is an invariant of $\mod{R[X]}$.

\begin{proposition}\label{monotone}
Given an $R[X]$-module $M_\phi$ and an $R[X]$-submodule $N_{\phi\restriction_N}$, then
\begin{enumerate}[\rm (1)]
\item $\widetilde{\ent}_v(\phi) \geq \widetilde{\ent}_v(\phi\restriction_N)$;
\item $\widetilde{\ent}_v(\phi) \geq \widetilde{\ent}_v(\bar \phi)$, where $\bar \phi: M/N \to M/N$ is the map induced by $\phi$;
\item $\widetilde{\ent}_v(\phi,K) \geq \widetilde{\ent}_v(\phi\restriction_N,N\cap K)+\widetilde{\ent}_v(\bar \phi, (K+N)/N)$, for any $K\in\I_\phi(M)$. 
\end{enumerate}
\end{proposition}
\begin{proof}
We give an argument just for part (3). Since $N_{\phi\restriction_N}$ is an $R[X]$-submodule of $M_\phi$, $\phi(N) \leq N$, hence $T_{n}(\phi,K\cap N) \leq N$, for all $n$. 
\\
For any $n\geq 1$ consider the following short exact sequence:
\[
0\to \frac{T_{n}(\phi,K)\cap(K+N)}{K}\to \frac{T_{n}(\phi,K)}{K}\to \frac{T_{n}(\phi, K)+N}{K+N}\to 0.
\]
The module on the right-hand side is  $T_{n}(\bar\phi, \bar K)/\bar K$, where $\bar K=(K+N)/N$. Consider now the following inclusion:
\begin{align}\label{inclusion}
\frac{T_{n}(\phi,K\cap N)}{K\cap T_{n}(\phi,K\cap N)}&\cong \frac{T_{n}(\phi,K\cap N)+K}{K}\subseteq \frac{T_{n}(\phi,K)\cap(K+N)}{K}
\end{align}
and the following quotient:
\begin{equation}\label{quotient}
\frac{T_{n}(\phi,K\cap N)}{K\cap T_{n}(\phi,K\cap N)}\to \frac{T_{n}(\phi,K\cap N)}{K\cap N}\to 0 
\end{equation}
The two formulas \eqref{inclusion} and \eqref{quotient} imply that
\[
L_v\left( \frac{T_{n}(\phi, K\cap N)}{ K\cap N}\right)\leq L_v\left( \frac{T_{n}(\phi,K)\cap(K+N)}{K}\right).
\]
Hence,
\[
L_v\left(\frac{T_{n}(\phi,K)}{K}\right)\geq L_v\left( \frac{T_{n}(\phi,K\cap N)}{K\cap N}\right) + L_v\left(\frac{T_{n}(\bar\phi, \bar K)}{\bar K}\right).
\]
Dividing by $n$ and taking the limit, we get $\ient(\phi,K)\geq\ient(\phi,K\cap N)+\ient(\bar\phi,\bar K)$.
\end{proof}

\subsection{Comparison with the valuation entropy $\ent_v$}

In this subsection we want to compare the valuation entropy $\ent_v$ with the intrinsic valuation entropy $\ient$. Given an $R[X]$-module $M_\phi$, since any submodule of finite valuation length is $\phi$-inert, there is an inclusion $\Fin(M)\subseteq \I_\phi(M)$; using Corollary \ref{inf_coro} (2), we deduce  the following inequality 
\[
\ent_v(\phi)\leq \widetilde{\ent}_v(\phi).
\] 
Our goal is to show that, if $M$ is a torsion module, the converse inequality also holds, analogously to what happens in the Abelian groups case. To prove this fact, we need the following technical lemma, which elaborates on \cite[Lem.\,4.1]{SV} with the due modifications.

\begin{lemma} \label{fingen}
Given an $R[X]$-module $M_\phi$, $H\in\I_\phi(M)$ and a real number $\epsilon > 0$:
\begin{enumerate}[\rm (1)]
\item there exists a finitely generated submodule $F$ of $H$ such that, for any $n \geq 1$,
\[
L_v\left(\frac{T_n(\phi, H)}{H}\right) - L_v\left(\frac{T_n(\phi, F)+H}{H}\right) < n \epsilon;
\]
\item if $M$ is torsion, then $\widetilde{\ent}_v(\phi,H) - \widetilde{\ent}_v(\phi, F) \leq \epsilon$.
\end{enumerate}
\end{lemma}
\begin{proof}
(1) By  \cite[Lem.\,2.1]{SV}, there exists a finitely generated submodule $K/H$ of $(H+ \phi H)/H$ satisfying the inequality $L_v((H+ \phi H)/H) - L_v(K/H) < \epsilon$, and $K/H$ is necessarily of the form $(H+ \phi F)/H$, with $F$ a finitely generated submodule of $H$. This gives the claim for $n=2$. Assuming the claim true for $n \geq 2$, i.e. assuming that $L_v(T_n(\phi, H)/(T_n(\phi, F)+H)) < n \epsilon$, we  prove it for $n+1$. Indeed,
\begin{align*}
 L_v\left(\frac{H+\phi T_n(\phi, H)}{H+\phi H+\phi T_n(\phi, F)}\right)&\leq L_v\left(\frac{\phi T_n(\phi, H)}{\phi H+\phi T_n(\phi, F)}\right)\\
&\leq L_v\left(\frac{T_n(\phi, H)}{H+T_n(\phi, F)}\right)<n\epsilon
\end{align*}
On the other hand, 
\begin{align*}
L_v\left(\frac{H+\phi H+\phi T_n(\phi, F)}{H+\phi F + \phi T_{n}(\phi, F)}\right) &\leq L_v\left(\frac{H+\phi H}{H+\phi F}\right)\\
&\leq L_v(\phi H/\phi F)\leq L_v(H/F)< \epsilon
\end{align*}
Since $H+\phi F + \phi T_{n}(\phi, F)= H+T_{n+1}(\phi, F)$, applying the above inequalities to the following exact sequence:
\[
0 \to \frac{H+\phi H+\phi T_n(\phi, F)}{H+T_{n+1}(\phi, F)} \to \frac{T_{n+1}(\phi, H)}{H+T_{n+1}(\phi, F)} \to \frac{T_{n+1}(\phi, H)}{H+\phi H+\phi T_n(\phi, F)} \to 0
\]
we obtain  the desired inequality for $n+1$.

\smallskip\noindent
(2) Being $M$ torsion, the finitely generated submodule $F \leq H$ has finite valuation length, hence it is  $\phi$-inert in $M$. For each $n \geq 1$, using the exact sequence
\[
0 \to \frac{T_n(\phi,F) \cap H}{F} \to \frac{T_n(\phi,F) }{F}\to \frac{T_n(\phi,F)+H }{H}\to 0
\]
we get:
\begin{align*}
&L_v\left(\frac{T_n(\phi, H)}{H}\right) - L_v\left(\frac{T_n(\phi, F)}{F}\right)\\
&= L_v\left(\frac{T_n(\phi, H)}{H}\right) - \left(L_v\left(\frac{(T_n(\phi, F)+H)}{H}\right)+ L_v\left(\frac{(T_n(\phi, F) \cap H)}{F}\right)\right)\\
&\leq L_v\left(\frac{T_n(\phi, H)}{H}\right) - L_v\left(\frac{T_n(\phi, F)+H}{H}\right) < n \epsilon.
\end{align*}
Dividing by $n$ and passing to the limit we get $\widetilde{\ent}_v(\phi,H) - \widetilde{\ent}_v(\phi, F) \leq \epsilon$.
\end{proof}

We can now prove the announced result for torsion modules.

\begin{proposition}\label{torsionequal}
If $\phi:M \to M$ is an endomorphism of a torsion $R$-module $M$, then $\ent_v(\phi)=\widetilde{\ent}_v(\phi)$.
\end{proposition}
\begin{proof} The inequality $\ent_v(\phi)\leq \widetilde{\ent}_v(\phi)$ is always true. In order to prove the converse inequality, it is enough to prove that, given a $\phi$-inert submodule $H$ of $M$, $\widetilde{\ent}_v(\phi, H) \leq \ent_v(\phi)$. Fix a real number $\epsilon >0$ and choose, according to Lemma \ref{fingen}, a finitely generated submodule $F$ of $H$  (hence $L_v(F)< \infty$, being $M$ torsion) such that $L_v((H+ \phi H)/H) - L_v((H+ \phi F)/H) < \epsilon$. By Lemma \ref{fingen} (2), $\widetilde{\ent}_v(\phi,H) - \widetilde{\ent}_v(\phi, F) < \epsilon$. Furthermore, by Corollary \ref{inf_coro} (2), $\widetilde{\ent}_v(\phi, F) = \ent_v(\phi, F)$.
Therefore, for each $\epsilon >0$ there exists a finitely generated submodule $F$ of $H$ such that 
$$\widetilde{\ent}_v(\phi, H) \leq \ent_v(\phi, F)+ \epsilon \leq  \ent_v(\phi)+ \epsilon. $$
As $\epsilon$ was arbitrary, we get that $\widetilde{\ent}_v(\phi, H) \leq \ent_v(\phi)$.
\end{proof}

\section{Some tools for the computation of  entropy}

\subsection{The intrinsic valuation entropy is upper continuous}

In the following proposition, as a consequence of Lemma \ref{fingen}, we prove that $\widetilde{\ent}_v$ is continuous with respect to direct limits of $\phi$-invariant submodules. 

\begin{proposition}\label{continuous}
Let $\phi: M \to M$ be an  endomorphism of an $R$-module which is the direct limit of a family of $\phi$-invariant submodules $\{ M_i : i \in I \}$. Then 
\[
\widetilde{\ent}_v(\phi)=\sup_i  \widetilde{\ent}_v(\phi _i),
\] 
where $\phi_i = \phi \restriction_{M_i}$ for all $i \in I$.
\end{proposition}
\begin{proof}
The inequality $ \geq$ is an immediate consequence of Proposition \ref{monotone}. In order to prove the inequality $ \leq$, let $H$ be a $\phi$-inert submodule of $M$. Then $H \cap M_i$ is $\phi_i$-inert in $M_i$ for all $i$, by Lemma \ref{restriction}. By Lemma \ref{fingen} (1), there exists a finitely generated submodule $F$ of $H$ such that 
\[
L_v\left(\frac{T_n(\phi, H)}{H}\right) - L_v\left(\frac{T_n(\phi, F)+H}{H}\right) < n \epsilon
\]
for each $n \geq 1$. Then $F$ is contained in $M_i$ for a certain $i$. Consequently $F \leq H \cap M_i$ and $T_n(\phi,F) \leq T_n(\phi_i,H \cap M_i)$ for all $n \geq 1$. As in the proof of \cite[Lem.\,3.14]{DGSV} one can show that
\[
T_n(\phi_i,H \cap M_i)=(H \cap M_i)+T_n(\phi,F),
\]
consequently we have the equality
\[
\frac{(H \cap M_i)+T_n(\phi, F)}{H \cap M_i} = \frac{T_n(\phi_i,H \cap M_i)}{H \cap M_i}
\]
Let us consider  the epimorphism:
\[
 \frac{(H \cap M_i)+T_n(\phi, F)}{H \cap M_i}\cong\frac{T_n(\phi, F)}{H \cap M_i \cap T_n(\phi, F)} \twoheadrightarrow  \frac{T_n(\phi, F)}{H \cap T_n(\phi, F)}\cong\frac{T_n(\phi, F)+H}{H} 
\]
which shows that $L_v(\frac{T_n(\phi, F)+H}{H}) \leq L_v(\frac{T_n(\phi, F)}{H \cap M_i \cap T_n(\phi, F)})=L_v(\frac{T_n(\phi_i,H \cap M_i)}{H \cap M_i})$. In conclusion, for each $n \geq 1$:
\[
L_v\left(\frac{T_n(\phi, H)}{H}\right) \leq L_v\left(\frac{T_n(\phi, F)+H}{H}\right) + n \epsilon \leq  L_v\left(\frac{T_n(\phi_i,H \cap M_i)}{H \cap M_i}\right) + n \epsilon.
\]
Dividing by $n$ and passing to the limit, we get:
$$\widetilde{\ent}_v(\phi,H) \leq \widetilde{\ent}_v(\phi_i,H \cap M_i) + \epsilon.$$
Being $\epsilon$ arbitrary, we deduce that $\widetilde{\ent}_v(\phi,H) \leq \widetilde{\ent}_v(\phi_i,H \cap M_i)$. From this inequality the conclusion easily follows.
\end{proof}

As an immediate consequence we deduce the upper continuity of $\widetilde{\ent}_v$.

\begin{corollary}\label{UC}
The intrinsic valuation entropy $\widetilde{\ent}_v$ is an upper continuous invariant of $\mod {R[X]}$.
\end{corollary}
\begin{proof}
Let $\phi: M \to M$ be an  endomorphism of the $R$-module $M$. Then $M$, viewed as an $R[X]$-module, is the direct union of the family of its finitely generated $R[X]$-submodules, which are exactly the trajectories $T(\phi,F)$ for $F$ a finitely generated $R$-submodule of $M$. Then, Proposition \ref{continuous} ensures that $\widetilde{\ent}_v(\phi)=\sup_{F \in {\mathcal F}(M)}  \widetilde{\ent}_v(\phi\restriction _{T(\phi,F)})$.
\end{proof}

\subsection{The Limit-Free Formula}\label{lim_free_subs}

In this subsection we prove that, given an $R[X]$-module $M_\phi$, the limit computation in the definition of the intrinsic valuation entropy of $\phi$ can be avoided (see Proposition \ref{nolimit}). This fact generalizes \cite[Prop.\,5.2]{SV}, that proves the same formula in case $M$ is torsion.

\begin{lemma}\label{in_and_inv}
Let $M_\phi$ be an $R[X]$-module. Given a $\phi$-inert submodule $K$ of $M$, the following statements hold:
\begin{enumerate}[\rm (1)]
\item $\widetilde\ent_v(\phi,A(\phi,K))=L_v(A(\phi,K)/\phi^{-1}A(\phi,K))$;
\item $\widetilde\ent_v(\phi,K)=\widetilde\ent_v(\phi,A(\phi,K))$.
\end{enumerate}
\end{lemma}
\begin{proof}
Part (1) follows from Proposition \ref{inf} and parts (3) and (4) of Proposition \ref{prop_K}: 
\[
\widetilde\ent_v(\phi,A(\phi,K))=\lim_{n\to\infty}L_v\left(\frac{T_{n+1}(\phi,A(\phi,K))}{T_n(\phi,A(\phi,K))}\right)=L_v\left(\frac{A(\phi,K)}{\phi^{-1}A(\phi,K)}\right).
\] 

\smallskip\noindent
(2) For all $n \geq 1$ let $S_n=K \cap \phi^{-1}A_n(\phi,K)$, and $S_\infty=\bigcup_{n}S_n=K \cap \phi^{-1}A(\phi,K)$. Clearly we have the isomorphisms 
\[
\frac{K}{S_n}\cong \frac{K+\phi^{-1}A_n(\phi,K)}{\phi^{-1}A_n(\phi,K)} = \frac{A_{n+1}(\phi,K)}{\phi^{-1}A_{n}(\phi,K)}\qquad\text{and}\qquad\frac{K}{S_\infty} \cong \frac{A(\phi,K)}{\phi^{-1}A(\phi,K)}.
\]
The claim now follows from the following equalities:
\begin{align*}
\widetilde\ent_v(\phi,K)&=\lim_{n \to \infty} L_v(T_{n+1}(\phi,K)/T_n(\phi,K))=\lim_{n \to \infty} L_v\left(\frac{A_{n+1}(\phi,K)}{\phi^{-1}A_n(\phi,K)}\right)\\
&=\lim_{n \to \infty} L_v(K/S_n)=L_v(K/S_{\infty})=L_v(A(\phi,K)/\phi^{-1}A(\phi,K))\,,
\end{align*}
where the second equality uses part (2) of Lemma \ref{lem_K_n}, while the fourth equality uses \cite[Lem.\,2.2]{SV}. 
\end{proof}

It follows from part (1) of the above lemma that, when we consider a $\phi$-inert submodule $K$ such that $\phi^{-1}K\subseteq K$ (and thus $K=A(\phi,K)$), we do not need to compute limits to evaluate the intrinsic valuation entropy with respect to $K$. In fact, it is a consequence of part (2) of the lemma that these $\phi$-inert submodules alone suffice to compute the intrinsic valuation entropy.

\begin{proposition}[Limit-Free Formula] \label{nolimit}
Let $M_\phi\in \mod {R[X]}$. Then
$$\widetilde\ent_v(M_\phi)=\sup\{L_v(N/\phi^{-1}N) : \ \text{$N\leq M$ is $\phi$-inert  and $\phi^{-1}N\subseteq N$}\}\,.$$
\end{proposition}
\begin{proof}
Given a $\phi$-inert submodule $N\leq M$ such that $\phi^{-1}N\subseteq N$, Lemma \ref{in_and_inv} shows that $\widetilde\ent_v(\phi,N)=L_v(N/\phi^{-1}N)$, so that the inequality ``$\geq$" in the statement is clear. For the converse inequality, let $K\leq M$ be a $\phi$-inert submodule. Then, by Proposition \ref{prop_K} and Lemma \ref{in_and_inv}, $A(\phi,K)$ is $\phi$-inert, $\phi^{-1}A(\phi,K)\subseteq A(\phi,K)$  and 
$\widetilde\ent_v(\phi,A(\phi,K))=\widetilde{\ent}_v(\phi, K).$
\end{proof}

\section{The Addition Theorem}

This section is devoted to the proof of the following result.

\begin{theorem}\label{AT}
The intrinsic valuation entropy 
$$\widetilde\ent_v\colon \mod {R[X]}\to \R_{\geq0}\cup\{\infty\}$$ 
is a length function.
\end{theorem}

As we have already verified in Corollary \ref{UC} that $\ient$ is an upper continuous invariant, we just need to prove its additivity. Indeed, given an $R[X]$-module $M_\phi$ and an $R[X]$-submodule $N_{\phi\restriction_{N}}$, we have to verify that
\begin{equation}\label{eq_AT}
\ient(\phi)=\ient(\phi\restriction_N)+\ient(\bar \phi)
\end{equation}
where $\bar\phi\colon M/N\to M/N$ is the induced map. The proof of this fact is quite involved so we divide it in several steps. In particular, in Subsection  \ref{final_AT_subs} we use the Limit Free Formula to show that $\ient$ is sub-additive (that is, the inequality $\leq$ in \eqref{eq_AT}) and, in Subsection \ref{first_inequality}, we verify the super-additivity of $\ient$, ending the proof of Theorem \ref{AT}.

\subsection{Sub-additivity of $\widetilde\ent_v$}\label{final_AT_subs}

Given an $R$-module $M$ and an endomorphism $\phi\colon M\to M$, the following submodule
\[
\ker_\infty(\phi):=\bigcup_{n\in\N}\ker(\phi^n)
\]
is called the {\em hyperkernel of $\phi$}. It is a $\phi$-invariant submodule of $M$ and, in fact, it is the smallest $\phi$-invariant submodule such that the induced endomorphism 
\[
\bar\phi\colon M/\ker_\infty(\phi)\to M/\ker_\infty(\phi)
\]
is injective. Consider now the ring of {\em Laurent polynomials} $R[X^{\pm1}]$, which can be viewed as the localization of $R[X]$ at the multiplicative set $\{X^n:n\in\N\}$. Consider the tensor product 
\[
{\frak M}_{\Phi}=M_\phi\otimes_{R[X]}R[X^{\pm1}] .
\]
As an $R[X]$-module, ${\frak M}_{\Phi}$ can be viewed as the direct limit of the following direct system:
\begin{equation}\label{tensor_is_direct_union}
M_\phi\overset{\phi}{\to}M_\phi\overset{\phi}{\to}M_\phi\overset{\phi}{\to}M_\phi\to \cdots
\end{equation}
It is not difficult to show that the kernel of the canonical map $M_\phi \to {\frak M}_{\Phi}$ is precisely $\ker_{\infty}(\phi)$ and that, in fact, identifying $(M/\ker_\infty(\phi))_{\bar\phi}$ as an $R[X]$-submodule of ${\frak M}_{\Phi}$, we have the following isomorphism in $\mod {R[X]}$:
\[
{\frak M}_{\Phi}\cong \bigcup_{n\in\N}\Phi^{-n}((M/\ker_\infty(\phi))_{\bar\phi}).
\]

\begin{lemma}\label{prop_reduction_invertible}
Let $M_\phi$ be an $R[X]$-module. The following equalities hold:
\begin{enumerate}[\rm (1)]
\item $\widetilde\ent_v(M_\phi)=\widetilde\ent_v((M/\ker_\infty(\phi))_{\bar\phi})$;
\item $\widetilde\ent_v(M_\phi)=\widetilde\ent_v(M_\phi\otimes_{R[X]}R[X^{\pm1}])$.
\end{enumerate}
\end{lemma}
\begin{proof}
(1) By Proposition \ref{nolimit},
\[
\widetilde\ent_v(M_\phi)=\sup\{L_v( N/\phi^{-1}N) : \ \text{$N\leq M$ is $\phi$-inert  and $\phi^{-1}N\subseteq N$}\}.
\]
Given a $\phi$-inert submodule $N\leq M$ such that $\phi^{-1}N\subseteq N$, it is clear that $\ker_{\infty}(\phi)\leq N$, so that, letting $\bar N=N/\ker_\infty(\phi)$,  we get 
\[ L_v( N/\phi^{-1}N)=L_v(\phi \bar N/\bar\phi^{-1}\bar N)\leq \widetilde\ent_v((M/\ker_\infty(\phi))_{\bar\phi}). 
\]
Thus, $\widetilde\ent_v(M_\phi)\leq \widetilde\ent_v((M/\ker_\infty(\phi))_{\bar\phi})$. The converse inequality follows from Proposition \ref{monotone}.

\smallskip\noindent
(2) We have proved in Proposition \ref{continuous} that $\ient$ is continuous with respect to direct unions. Hence, by the description of ${\frak M}_{\Phi}$ as a direct union of copies of $(M/\ker_{\infty}(\phi))_{\bar \phi}$ and by part (1), we get
$\ient({\frak M}_{\Phi})=\ient((M/\ker_\infty(\phi))_{\bar\phi})=\ient(M_\phi)$.
\end{proof}

\begin{lemma}\label{AT_for_inv}
Let $0\to N_{\phi\restriction_N}\to M_\phi \to (M/N)_{\bar\phi}\to 0$ be a short exact sequence in $\mod {R[X^{\pm 1}]}$. Then, 
$$\widetilde\ent_v(M_\phi)\leq\widetilde\ent_v(N_{\phi\restriction_N})+\widetilde\ent_v((M/N)_{\bar\phi})\,.$$
\end{lemma}
\begin{proof}
Let $T\leq M$ be a $\phi^{-1}$-invariant $\phi$-inert submodule; let $T'=T\cap N$ and $\bar T=(T+N)/N$. Clearly $T'$ and $\bar T$ are $\phi^{-1}$-invariant and $\bar \phi^{-1}$-invariant, respectively. Using the equality $N=\phi^{-1}N$ and the fact that $\phi^{-1}$ commutes with intersection of submodules, we get the following isomorphisms:
$$\frac{ T'}{\phi^{-1}T'} \cong \frac{T \cap(\phi^{-1}T+N)}{\phi^{-1}T} \ \ , \ \ \frac{T}{T \cap(\phi^{-1}T+N)} \cong \frac{ \bar T}{\bar \phi^{-1} \bar T} \ .$$
From these isomorphisms we obtain the exact sequence
\begin{equation}\label{sostituto_snake} 
0\to  \frac{ T'}{\phi^{-1}T'}\to \frac{T}{\phi^{-1}T}\to \frac{ \bar T}{\bar \phi^{-1} \bar T}\to 0
\end{equation}
that, together with Proposition \ref{nolimit}, shows that
\[
L_v(T/\phi^{-1}T)=L_v(T'/\phi^{-1}T')+L_v(\bar T/\bar \phi^{-1}\bar T)\leq \widetilde\ent_{v}({\phi\restriction_N})+\widetilde\ent_v(\bar \phi).
\]
Being $T$ arbitrary, we obtain that $\widetilde\ent_v(\phi)\leq \widetilde\ent_{v}({\phi\restriction_N})+\widetilde\ent_v(\bar \phi)$.
\end{proof}

In order to complete the proof of the sub-additivity, we must pass form $R[X^{\pm 1}]$-modules considered in Lemma \ref{AT_for_inv} to $R[X]$-modules.

\begin{proposition}\label{first_half_AT_prop}
Consider a short exact sequence in $\mod {R[X]}$
\[
0\to N_{\phi\restriction_N}\to M_\phi\to (M/N)_{\bar\phi}\to 0.
\]
Then, $\widetilde\ent_v(\phi)\leq \widetilde\ent_v(\phi\restriction_N)+\widetilde\ent_v(\bar\phi)$. 
\end{proposition}
\begin{proof}
Since $R[X^{\pm 1}]=\bigcup_n X^{-n}R[X]$ is a flat $R[X]$-module, $-\otimes_{R[X]}R[X^{\pm1}]$ is an exact functor, so the following sequence is exact in $\mod {R[X^{\pm1}]}$:
\[
0\to N_{\phi\restriction_N}\otimes_{R[X]}R[X^{\pm1}]\to M_\phi\otimes_{R[X]}R[X^{\pm1}]\to (M/N)_{\bar\phi}\otimes_{R[X]}R[X^{\pm1}]\to 0.
\]
By Lemma \ref{prop_reduction_invertible}, 
\[
\ient(N_{\phi\restriction_N})=\ient(N_{\phi\restriction_N}\otimes_{R[X]}R[X^{\pm1}]) \ \ \  ,  \ \  \ient(M_{\phi})=\ient(M_\phi\otimes_{R[X]}R[X^{\pm1}]) \ \ ,
\]
\[
\ient((M/N)_{\bar\phi})=\ient((M/N)_{\bar\phi}\otimes_{R[X]}R[X^{\pm1}])\ \ ,
\]
while, by Lemma \ref{AT_for_inv},
\[
\ient(M_\phi\otimes_{R[X]}R[X^{\pm1}])\leq \ient(N_{\phi\restriction_N}\otimes_{R[X]}R[X^{\pm1}])+ \ient((M/N)_{\bar\phi}\otimes_{R[X]}R[X^{\pm1}]).
\]
Hence, the desired inequality follows.
\end{proof}

\subsection{Super-additivity of $\widetilde\ent_v$}\label{first_inequality}

The next technical lemma deals with the intrinsic valuation entropy with respect to finitely generated $\phi$-inert submodules; notice that part (2) is particularly important, as it says that, given such a submodule $H$ of $M$, the quantity $\widetilde\ent_v(\phi,H)$ does not depend on $H$ but only on $T(\phi,H)$, that is, for any other finitely generated a $\phi$-inert submodule $H'$ with $T(\phi,H)=T(\phi,H')$, one has $\widetilde\ent_v(\phi,H)=\widetilde\ent_v(\phi,H')=\widetilde\ent_v(\phi\restriction_{T(\phi,H)})$.

\begin{lemma}\label{F,F'}
Consider an $R[X]$-module $M_\phi$ and let $H' \subseteq H\in \I_\phi(M)$,
\begin{enumerate}[\rm (1)]
\item if $H/H'$ is finitely generated, then $\ient(\phi,H')\leq \ient(\phi,H)$;
\item if $H$ is finitely generated and $M=T(\phi,H)$, then $\ient(\phi)= \ient(\phi,H)$.
\end{enumerate}
\end{lemma}
\begin{proof}
(1) Consider the following exact sequence
\[
0\to \frac{H\cap T_n(\phi,H')}{H'}\to \frac{T_n(\phi,H')}{H'}\to \frac{T_n(\phi,H)}{H}.
\]
As both $H$ and $H'$ are $\phi$-inert, all the modules appearing in the above sequence are $L_v$-finite. This implies that, for every $n\geq 1$, the following inequality holds
\[
L_v \left(\frac{T_n(\phi,H')}{H'}\right)-L_v \left(\frac{H\cap T_n(\phi,H')}{H'}\right)\leq L_v\left(\frac{T_n(\phi,H)}{H}\right) .
\]
As remarked before Lemma \ref{restriction}, $T(\phi,H')/H'$ is a torsion module, hence $(H/H')\cap (T(\phi,H')/H')$ is a torsion submodule of the finitely generated module $(H/H')$; so, by Lemma \ref{fin_gen_tor_fin},  $L_v((H/H')\cap (T(\phi,H')/H'))<\infty$. For every $n\geq 1$ there is an inclusion $(H\cap T_n(\phi,H'))/H'\subseteq (H/H')\cap (T(\phi,H')/H')$, which implies 
\[
0\leq\lim_{n\to\infty}\frac{L_v((H\cap T_n(\phi,H'))/H')}{n}\leq \lim_{n\to\infty}\frac{L_v((H/H')\cap (T(\phi,H')/H'))}{n}=0.
\]
Hence we obtain:
\begin{align*}
\widetilde{\ent}_v(\phi,H')&=\lim_{n\to\infty}\frac{L_v(T_n(\phi,H')/H')}{n}\\
&=\lim_{n\to\infty}\frac{L_v(T_n(\phi,H')/H')}{n}-\lim_{n\to\infty}\frac{L_v((H\cap T_n(\phi,H'))/H')}{n}\\
&=\lim_{n\to\infty}\frac{L_v(T_n(\phi,H')/H')-L_v((H\cap T_n(\phi,H'))/H')}{n}\\
&\leq \lim_{n\to\infty}\frac{L_v(T_n(\phi,H)/H)}{n}=\widetilde{\ent}_v(\phi,H).
\end{align*}

\smallskip\noindent
(2) Given a $\phi$-inert submodule $K$ of $M$, we have to prove that $\widetilde \ent_v(\phi,K)\leq \widetilde \ent_v (\phi,H)$. First of all, we can suppose that $H\leq K$, as $\widetilde \ent_v(\phi,K)\leq \widetilde \ent_v(\phi,H+K)$ by part (1) (being $H+K$ $\phi$-inert and $(H+K)/K$ finitely generated). By Lemma \ref{fingen}, there exists a finitely generated submodule $F$ of $K$ such that, for any $n \geq 1$,
\[
L_v\left(\frac{T_n(\phi,K)}{K}\right)  < L_v\left(\frac{T_n(\phi, F)+K}{K}\right)+ n \epsilon.
\]
Since $M=T(\phi,H)$ and $F$ is finitely generated, there exists $k\geq 1$ such that $F\leq T_k(\phi,H)$ and so
\[
L_v\left(\frac{T_n(\phi, F)+K}{K}\right)+ n \epsilon\leq L_v\left(\frac{T_{n+k}(\phi,H)+K}{K}\right)+ n \epsilon
\]
for every $n\geq 1$. Since $(K+T_{n+k}(\phi,H))/K$ is a quotient of $T_{n+k}(\phi,H)/H$, as we supposed that $H\leq K$, 
the inequality 
\[
L_v\left(\frac{T_n(\phi,K)}{K}\right)\leq L_v\left(\frac{T_{n+k}(\phi,H)}{H}\right)+n\epsilon
\]
holds for every $n\geq 1$. 
Hence, $\widetilde\ent_v(\phi,K)\leq \widetilde\ent_v(\phi, T_k(\phi,H))+\epsilon$ and so, by Corollary \ref{inf_coro} (1), $\widetilde\ent_v(\phi,K)\leq \widetilde\ent_v(\phi, T_k(\phi,H))+\epsilon=\widetilde\ent_v(\phi, H)+\epsilon$. We can now conclude by the arbitrariness of $\epsilon$. 
\end{proof}  

As a consequence of Lemma \ref{F,F'} we derive one direction of the following proposition.

\begin{proposition}\label{entropy_fg_coro}
Let $M_\phi=T(\phi,F)$ be a finitely generated $R[X]$-module, with $F$ a finitely generated $R$-submodule of $M$. Then $\ient(\phi)<\infty$ if and only if $\rk_R(M)<\infty$.
\end{proposition}
\begin{proof}
Assume $\ient(\phi)<\infty$. The factor module $(M/t(M))_{\bar \phi}$ is finitely generated and  $\rk_R(M) =\rk_R(M/t(M))$. By Proposition 2.11 (2), $\ient(\bar \phi)\leq \ient(\phi)$, so it is enough to prove that $\rk_R(M/t(M))<\infty$. Assume, by way of contradiction, that $\rk_R(M/t(M))=\infty$. Since $(M/t(M))_{\bar \phi}$ is a finite sum of cyclic trajectories, at least one of them, say $T(\bar\phi, x)$, must have infinite rank. Then necessarily $T(\bar\phi, x)=\oplus_{n \geq 0}\bar\phi^nxR$, and clearly
$ \ient(\bar\phi\restriction{T(\bar\phi, x)})=\infty$, so $\ient(\bar\phi)=\infty$ by Proposition 2.11 (1), absurd. 

Conversely, assume that $\rk_R(M)<\infty$. Then there exists an index $k$ such that $rk_R(T_m(\phi,F))=\rk_R(M)$ for all $m \geq k$. This implies that 
$T_{k+1}(\phi,F)/T_k(\phi,F)$ is torsion, so it is of finite valuation length, being finitely generated. This implies that $T_k(\phi,F)$ is $\phi$-inert. From Lemma \ref{F,F'} we derive that $\ient(\phi)=\ient(\phi,T_k(\phi,F))$, which is obviously finite.
\end{proof}

Below is another application of Lemma \ref{F,F'}, where we use the fact that finitely generated torsion-free $R$-modules are free (see \cite[V.2.8]{FS1}).

\begin{proposition}\label{freeinfr}
Let $M_\phi$ be an $R[X]$-module such that, as an $R$-module, $M$ is torsion-free and of finite rank. If $F$ is  a finitely generated $R$-submodule of maximum rank, then $F$ is $\phi$-inert and 
\[
\widetilde{\ent}_v(\phi)=\widetilde{\ent}_v(\phi\restriction_{T(\phi,F)})=\widetilde{\ent}_v(\phi,F).
\]
\end{proposition}
\begin{proof}
We know that $F$ is $\phi$-inert from Lemma \ref{Kinert}.
The divisible envelope $D(M)$ of $M$ is the same as the divisible envelope of $T(\phi,F)$. Denote by $\widetilde\phi\colon D(M)\to D(M)$ the extension of $\phi$ to the divisible envelope. With the obvious modifications, the same argument of \cite[Cor.\,3.15]{DGSV} shows that $\ient(\phi)=\ient(\widetilde\phi)=\ient(\phi\restriction_{T(\phi,F)})$. It is now a consequence of Proposition \ref{F,F'} that $\ient(\phi\restriction_{T(\phi,F)})=\ient(\phi,F)$.
 \end{proof}
 
We are in position to prove AT for torsion-free modules of finite rank.

\begin{corollary}\label{tors_free_case_AT_coro}
Let $M$ be a torsion-free $R$-module of finite rank, and let $\phi\colon M\to M$ be an endomorphism. Given a $\phi$-invariant submodule $N$ of $M$ such that $M/N$ is torsion-free, we have that
\[
\ient(\phi)=\ient(\phi\restriction_N)+\ient(\bar \phi)
\]
where $\bar\phi\colon M/N\to M/N$ is the induced map. 
\end{corollary}
\begin{proof}
Let $F\leq M$ be a finitely generated $R$-submodule of maximum rank. Then, $\bar F=(F+N)/N$ is finitely generated and torsion free, so it is free, showing that $F\cong (F\cap N)\oplus \bar F$ and so also $F\cap N$ is free. Notice also that $\rk_R(F\cap N)=\rk_R(N)$ and $\rk_R(\bar F)=\rk_R(M/N)$. By Proposition \ref{monotone} (3) and Proposition \ref{freeinfr},
\[
\ient(\phi)=\ient(\phi,F)\geq \ient(\phi, F\cap N)+\ient(\bar\phi, \bar F)=\ient(\phi\restriction_N)+\ient(\bar \phi).
\]
The converse inequality is proved in Proposition \ref{first_half_AT_prop}.
\end{proof}

The torsion case of AT could be deduced by one of the main results of \cite{SV}, since in that case $\ient= \ent_v$, by Proposition \ref{torsionequal}. But, for the sake of completeness, we prefer to give here a direct argument.

\begin{proposition}\label{tors_case_AT_coro}
Let $M$ be a torsion $R$-module, and let $\phi\colon M\to M$ be an endomorphism. Given a $\phi$-invariant submodule $N$ of $M$, we have that
\[
\ient(\phi)=\ient(\phi\restriction_N)+\ient(\bar \phi)
\]
where $\bar\phi\colon M/N\to M/N$ is the induced map. 
\end{proposition}
\begin{proof}
By Proposition \ref{first_half_AT_prop} we already know that
$
\ient(\phi)\leq\ient(\phi\restriction_N)+\ient(\bar\phi),
$
so it is enough to show the converse inequality. Let $K_1\leq N$ and $\bar K_2\leq M/N$ be finitely generated (so $L_v$-finite), and let $K\leq M$ be a finitely generated submodule such that $(K+N)/N=\bar K_2$ and $K_1\subseteq N\cap K$. By Proposition \ref{monotone} (3),  $\ient(\phi,K)\geq\ient(\phi,K\cap N)+\ient(\bar\phi,\bar K_2)$ and, since $\ient(\phi,N\cap K)=\ent_v(\phi,N\cap K)\geq \ent_v(\phi, K_1)=\ient(\phi,K_1)$, we get $\ient(\phi,K)\geq \ient(\phi,K_1)+\ient(\bar\phi,\bar K_2)$.
\end{proof}

The next proposition proves AT in the particular case when the $\phi$-invariant submodule is the torsion part. 

\begin{proposition}\label{entropy_fg_coro_wrt_tors}
Let $M_\phi$ be an $R[X]$-module such that $\rk_R(M)=k<\infty$. Then $\ient(\phi)=\ient(\phi\restriction_{tM})+\ient(\bar\phi)$, with $\bar\phi\colon M/tM\to M/tM$ the induced map.
\end{proposition}
\begin{proof}
By Proposition \ref{first_half_AT_prop} we already know that
$
\ient(\phi)\leq\ient(\phi\restriction_{tM})+\ient(\bar\phi),
$
so it is enough to show the converse inequality. 

We must prove that, fixed a $\phi$-inert submodule $K_1$ of $tM$ and a $\bar\phi$-inert submodule $\bar K_2$ of $M/tM$, there exists a $\phi$-inert submodule $K$ of $M$ such that $\ient(\phi,K)\geq \ient(\phi,K_1)+\ient(\bar\phi,\bar K_2)$. By \cite[Proposition 4.2]{SV}, we may assume that $K_1$ is finitely generated and, by Proposition \ref{freeinfr}, that $\bar K_2 \cong R^k$. Let now $K$ be a finitely generated submodule of $M$ such that $\bar K_2=(K+tM)/tM$. Adding $K_1$ to $K$, we can assume, without loss of generality, that $K_1 \leq K$.  Then $K$ is $\phi$-inert, by Lemma \ref{Kinert}. By Proposition \ref{monotone} (3),  $\ient(\phi,K)\geq\ient(\phi,tK)+\ient(\bar\phi,\bar K_2)$ and, since $\ient(\phi,tK)=\ent_v(\phi,tK)\geq \ent_v(\phi, K_1)=\ient(\phi,K_1)$, we get  $\ient(\phi,K)\geq \ient(\phi,K_1)+\ient(\bar\phi,\bar K_2)$.
\end{proof}

We have now all the ingredients needed to prove the general form of AT.

\begin{theorem}[Addition Theorem] \label{additiontheorem}
Consider the following short exact sequence of $R[X]$-modules:
\[
0\to N_{\phi\restriction_N}\to M_\phi\to (M/N)_{\bar \phi}\to 0.
\]
Then, $\ient(\phi)= \ient(\phi\restriction_N)+\ient(\bar \phi)$.
\end{theorem}
\begin{proof}
In view of Propositon \ref{first_half_AT_prop}, it is enough to prove the inequality $"\geq"$. If either $\ient(\phi\restriction_N)=\infty$ or $\ient(\bar \phi)=\infty$ then, by Proposition \ref{monotone}, $\ient(\phi)=\infty$ and so the desired equality holds. Hence, suppose that $\ient(\phi\restriction_N)<\infty$ and $\ient(\bar \phi)<\infty$. By Corollary \ref{UC}, these entropies are the supremum of the entropies of the restrictions to finitely generated submodules and so, given $\epsilon>0$, there exist finitely generated $R$-submodules $F_1\leq N$ and $\bar F_2\leq M/N$ such that
\[
\ient(\phi\restriction_N)\leq \ient(\phi\restriction_{T(\phi,F_1)})+\epsilon/2\qquad\ient(\bar \phi)\leq \ient(\bar\phi\restriction_{T(\bar\phi,\bar F_2)})+\epsilon/2.
\]
Let $F$ be a finitely generated $R$-submodule of $M$ such that $(F+N)/N=\bar F_2$ and $F_1\leq F\cap N$. Let $M'=T(\phi, F)$, $N'=N\cap T(\phi,F)$, and notice that $M'/N'\cong T(\bar\phi,\bar F_2)$. By the arbitrariness of $\epsilon>0$, it is enough to prove the first inequality below, since the other two inequalities follow by the above arguments:
\begin{align*}
\ient(\phi\restriction_{M'})&\geq \ient(\phi\restriction_{N'})+\ient(\bar\phi\restriction_{M'/N'})\\
&\geq \ient(\phi\restriction_{T(\phi,F_1)})+\ient(\bar\phi\restriction_{T(\bar\phi,\bar F_2)})\\
&\geq \ient(\phi\restriction_N)+\ient(\bar\phi)-\epsilon. 
\end{align*}
In other words, we have reduced our problem to the case when $M=T(\phi,F)$ is a finitely generated $R[X]$-module. Now, if $\rk_R(M)=\infty$, then $\ient(\phi)=\infty$ by Proposition \ref{entropy_fg_coro}, and so clearly $\ient(\phi)\geq \ient(\phi\restriction_N)+\ient(\bar \phi)$. Hence, suppose $\rk_R(M)=k<\infty$.

Let $K=tM+N$; then $K/tM$ is a $\bar\phi$-invariant submodule of $M/tM$, as well as its purification $(K/tM)^*= \{ \bar x \in M/tM \ | \ r \bar x \in K/tM \ {\rm for \ some} \ r \neq 0 \}$.

\noindent Let us list some short exact sequences of $R[X]$-modules on which $\ient$ is additive:

\begin{enumerate}
\item[\rm (1)] $0\to tM\to M\to M/tM\to 0$, by Proposition \ref{entropy_fg_coro_wrt_tors};
\item[\rm (2)]  $0\to tN\to N\to K/tM\to 0$, by Proposition \ref{entropy_fg_coro_wrt_tors};
\item[\rm (3)] $0\to t(M/N)\to M/N\to (M/N))/t(M/N) \to 0$, by Proposition \ref{entropy_fg_coro_wrt_tors};
\item[\rm (4)]  $0\to tN\to tM\to tM/tN\to 0$, by Proposition \ref{tors_case_AT_coro};
\item[\rm (5)] $0\to (K/tM)^*\to M/tM\to (M/K)/t(M/K)\to 0$, by Corollary \ref{tors_free_case_AT_coro}.
\end{enumerate}
Furthermore, we also need the following three observations, in which, as in the next part of the proof, by abusing notation, we eliminate the subscript of the endomorphism for $\Z[X]$-modules when we apply $\ient$.
\begin{enumerate}
\item[(6)] $\ient((K/tM)^*/(K/tM))=0$. By Proposition \ref{torsionequal} and \cite[Proposition 4.2]{SV}, it is enough to show that, for any finitely generated $R$-submodule $C$ of $(K/tM)^*$, 
\[
L_v\left(\frac{T(\bar \phi, C)+(K/tM)}{K/tM}\right)<\infty.
\] 
For that we use Lemma \ref{bounded_is_finite}. Indeed, we know that $\rk_R(M)=k<\infty$, so that $(K/tM)^*\leq M/tM\leq Q^k$; furthermore $\frac{T(\bar \phi, C)+(K/tM)}{K/tM}\leq \frac{(K/tM)^*}{K/tM}$ are torsion modules. Choose a finite set $\{c_1,\ldots,c_t\}$ of generators for $C$, then there is an element $r\neq 0$ in $R$ such that $rc_i\in K/tM$, for all $i$. Hence $r$ annihilates $\frac{T(\bar \phi, C)+(K/tM)}{K/tM}$, which is then a bounded sub-module of $Q^k/(K/tM)$ and Lemma \ref{bounded_is_finite} applies.
\item[(7)] $\ient((K/tM)^*)=\ient(K/tM)$. This follows by point (6) and by the sub-additivity of $\ient$;
\item[(8)] $\ient(t(M/N))=\ient(tM/tN)$. This follows by point (6) and by the fact that $t(M/N)/(K/N)$ embeds into $(K/tM)^*/(K/tM)$; in fact, $t(M/N)/(K/N)$ embeds into $t(M/K)$, which is isomorphic to 
\[
t((M/tM)/(K/tM))=(K/tM)^*/(K/tM).
\] 
Hence $\ient((K/tM)^*/(K/tM))=0$ implies $\ient(t(M/N)/(K/N))=0$ and consequently $\ient(t(M/N))=\ient(K/N)=\ient(tM/tN)$, since $tM/tN \cong K/N$.
\end{enumerate}

\noindent Now we can conclude the proof with the following series of equalities depending on the associated overset points:
\begin{align*}
\ient(M)&\overset{(1)}{=}\ient(tM)+\ient(M/tM) \\
&\overset{(4,5)}{=}\ient(tN)+\ient(tM/tN)+\ient((K/tM)^*)+\ient((M/K)/t(M/K)) \\
&\overset{(7)}{=}\ient(tN)+\ient(tM/tN)+\ient(K/tM)+\ient((M/N)/t(M/N)) \\
&\overset{(2,3)}{=}\ient(N)+\ient(tM/tN)+\ient(M/N)-\ient(t(M/N)) \\
&\overset{(8)}{=}\ient(N)+\ient(M/N) \qedhere
\end{align*}
\end{proof}

\section{The Intrinsic Algebraic Yuzvinski Formula }

In order to state the Intrinsic Algebraic Yuzvinski Formula (IAYF for short) in a simple way, we introduce the following terminology. If $\phi\colon Q^n \to Q^n$ is a linear transformation, where $Q$ denotes the field of quotients of the valuation domain $R$, let $p_\phi(X) \in Q[X]$ be the (monic) characteristic polynomial of $\phi$ over $Q$. Since $R$ is a valuation domain, there exists an element $s \in R$ of minimal value such that $sp_\phi(X) \in R[X]$; thus the polynomial $sp_\phi(X)$ is primitive (i.e., its content $c(sp_\phi(X))$ equals $R$) with leading coefficient $s$, and it is called the {\it characteristic polynomial of $\phi$ over $R$}.  Our goal in this section is to prove the following version of the IAYF.

\begin{theorem}[IAYF] \label{IAYF} 
Let $\phi: Q^n \to Q^n$ be a linear transformation. Then $\widetilde{\ent}_v(\phi)=v(s)$, where $s\in R$ is the leading coefficient of the characteristic polynomial of $\phi$ over $R$.
\end{theorem}

Notice that this formula shows that $\widetilde{\ent}_v(\phi)=0$ exactly when $p_\phi(X)\in R[X]$, that is, when $\phi$ is integral over $R$, while, if some $q_i \in Q \setminus R$, the intrinsic valuation entropy takes a positive value.

\smallskip
Let us briefly explain where the statement of Theorem \ref{IAYF} comes from. Recall that the Algebraic Yuzvinski Formula, proved in \cite{GV}, deals with another algebraic entropy for Abelian groups, denoted by $h$ and deeply investigated in \cite {DG}. It states that the entropy $h(\phi)$ of an endomorphism $\phi$ of a finite dimensional vector space over the rational field $\Q$ coincides with the Mahler measure of the characteristic polynomial of $\phi$ (we refer to \cite{GV,DGSV} for the notions of the entropy $h$ and of the Mahler measure of a rational polynomial). In fact, this connection between the values of $h$ Mahler measure reflects a dual property of the topological entropy on solenoidal endomorphisms (see, for example, \cite{LW})

The Intrinsic Algebraic Yuzvinski Formula for the intrinsic entropy  $\widetilde{\ent}$ proved in \cite{DGSV} (see also \cite{GV2,SV1}), states that the entropy $\widetilde{\ent}(\phi)$ coincides with $\log(s)$, where $s$ is the minimal common multiple of the denominators of the rational numbers appearing in the characteristic polynomial $p_\phi(X)$ of $\phi$ over $\Q$.  This $s$ is the minimal positive integer such that $sp_\phi(X)$ is a primitive polynomial of $\Z[X]$. This shows the strict analogy between this result and Theorem \ref{IAYF}.

\smallskip
The IAYF is a consequence of the following result, which is analogous to Lemma 3.3 in \cite{SV1}, with the due modifications.

\begin{proposition}\label{cyclictrajectory}
Let $M$ be a torsion-free $R$-module and $\phi\colon M \to M$ an endomorphism. Let $x \in M$ be an element generating a $\phi$-trajectory $T(\phi,x)$ of finite rank $n \geq 1$. Then,
\begin{enumerate}
\item[\rm (1)] $T_n(\phi,x)=\bigoplus_{0 \le i \leq n-1} \phi^ixR$;
\item[\rm (2)] there exists a polynomial $f(X) \in R[X]$ of degree $n$, with content $c(f(X))=R$ and leading coefficient $s \in R$, such that $f(\phi)(x)=0$;
\item[\rm (3)] $R[X]/(f(X)) \cong T(\phi,x)$;
\item[\rm (4)] $T_{k+1}(\phi,x)/T_k(\phi,x) \cong R/sR$ for every $k \geq n$;
\item[\rm (5)] if $s$ is a unit, then $T(\phi,x)=T_n(\phi,x)$ is free of rank $n$, otherwise the quotient $T(\phi,x)/T_n(\phi,x)$ is a uniserial divisible module isomorphic to $Q/R$.
\end{enumerate}
\end{proposition}
\begin{proof}
(1) Let $m \geq 1$ be the minimal positive integer such that $t \phi^m x \in T_m(\phi,x)$ for some $0 \neq t \in R$. Then $T_m(\phi,x)= \bigoplus_{1 \leq i \leq m-1} \phi^i xR$ and a simple computation shows that, for each $k \geq 1$,
$t^k \phi^{m+k-1}x \in T_m(\phi,x)$; consequently $T(\phi,x)/T_m(\phi,x)$ is a torsion module, so $T_m(\phi,x)$ has rank $m$. This shows that $m=n$.

\smallskip\noindent
(2) Choose an element $t \in R$ in such a way that  $t \phi^n x \in T_n(\phi,x)$, and let $t \phi^n x= \sum_{0 \leq i \leq n-1} r_i \phi^i x$. Let $r_j$ be such that $v(r_j) \leq v(r_i)$ for every $i \leq n-1$. Setting $s=tr_j^{-1}$, the polynomial $f(X)=s \phi^n X - \sum_{0 \leq i \leq n-1} r_j^{-1}r_i \phi^i X$ clearly satisfies $c(f(X))=R$ and $f(\phi)(x)=0$.

\smallskip\noindent
(3) Consider the ideal of $R[X]$:
$$K_x= \{ g(X) \ | \ g(\phi)=0 \}. $$
The map $R[X] \to T(\phi,x)$ sending $g(X)$ into $g(\phi)$ is surjective and has kernel $K_x$, hence $R[X]/K_x \cong T(\phi,x)$. Clearly the ideal $(f(X))$ is contained in $K_x$, so, in order to conclude, it is enough to prove that,
if $g(\phi)=0$, then $g(X)$ is a multiple of $f(X)$. If $g(\phi)=0$, then obviously $g(X)$ has degree $k \geq n$; without loss of generality we can assume $g(X)$ primitive, that is, $c(g(X))=R$. The division algorithm gives:
$$s^{k-n+1}g(X)=f(X)q(X)+r(X)$$
where $r(X)$ has degree $<n$. This implies that $r(X)=0$, since $r(\phi)=0$. Therefore $s^{k-n+1}g(X)=f(X)q(X)$, thus $s^{k-n+1}$ divides $c(f(X)q(X))=c(q(X))$ (see property (b) of the content in \cite[pg.\,7]{FS1}). Dividing the above equality by $s^{k-n+1}$ we derive that $g(X)$ is a multiple of $f(X)$, as desired.

\smallskip\noindent
(4) The proof goes by induction on $k \geq n$, the starting case $k=n$, which follows from (2). Let $k>n$ and consider  the cyclic factor module $T_{k+1}(\phi,x)/T_k(\phi,x) \cong R/I_k$, where $I_k$ is a non-zero ideal of $R$. By \cite[Lem.\,1.9]{SZ}, 
\[L_v(T_{k+1}(\phi,x)/T_k(\phi,x)) \leq L_v(T_{n+1}(\phi,x)/T_n(\phi,x))=v(s),\]
therefore $I_k \geq sR$.
Let us assume, by way of contradiction, that $I_k > sR$. Then there exists $t \in I_k \setminus sR$ such that a polynomial $g(X)$ of degree $k$ and leading coefficient $t$ satisfies $g(\phi)(x)=0$. By point (3), $g(X)$ is a multiple of $f(X)$, hence $t\in sR$, a contradiction.

\smallskip\noindent
(5) The first claim is obvious. If $v(s)>0$, then (4) implies that $T(\phi,x)/T_n(\phi,x)$ is isomorphic to the direct limit of the direct system $\{ R/s^nR \}_{n\geq 1}$, where the connecting maps $R/s^nR \to R/s^{n+1}R$ are induced by the multiplication by $s$, and this direct limit is isomorphic to $Q/R$, being $R$ archimedean.
\end{proof}

\begin{proof}[Proof of Theorem \ref{IAYF}.]
The $Q$-vector space $V=Q^n$, endowed with the structure of $Q[X]$-module induced by $\phi$, is finitely generated and torsion. Since $Q[X]$ is a PID, $V_\phi$ decomposes into the direct sum of cyclic $Q[X]$-modules:
$$V_\phi=V_1 \oplus \cdots \oplus V_r.$$
If $\phi_i$ denotes the restriction of $\phi$ to the submodule $V_i$, then for the characteristic polynomials over $Q$ we have the factorization:
$$p_\phi(X)= \prod_{1 \leq i \leq r}p_{\phi_i}(X).$$
Then the Addition Theorem ensures that it is enough to prove the result for the cyclic summands, that is, we can assume that $V_\phi$ is a cyclic $Q[X]$-module. Then there exists an element $x \in V$ such that:
$$V=xQ \oplus \phi xQ \oplus \cdots \oplus \phi^{n-1} xQ.$$
Let $F=\bigoplus_{0\leq i \leq n-1}  \phi^{i} xR$ be the free partial $n$-th trajectory of $x$ of full rank in $V$. Then, by Proposition \ref{freeinfr}, $F$ is $\phi$-inert in $V$ and $\widetilde{\ent}_v(\phi)=\widetilde{\ent}_v(\phi,F)$. Now  Proposition \ref{inf} ensures that
\[
\widetilde{\ent}_v(\phi,F)=\inf_k  L_v\left(\frac{T_{k+1}(\phi, F )}{T_k(\phi, F )}\right).
\]
But for each $k \geq 1$ we have that $T_k(\phi,F)=T_{n+k-1}(\phi,x)$, therefore Proposition \ref{cyclictrajectory} (4) implies that $T_{k+1}(\phi,F)/T_k(\phi,F) \cong R/sR$ for every $k \geq 1$. We deduce that
$\widetilde{\ent}_v(\phi)=L_v(R/sR)=v(s)$, that gives the proof, since the polynomial $f(X)$ of Proposition \ref{cyclictrajectory} (2) coincides, in case $V$ is a cyclcic $\phi$-trajectory, with the characteristic polynomial of $\phi$ over $R$.
\end{proof}

\section{The Uniqueness Theorem }

The next result it is used at the beginning of the proof of \cite[Thm.\,7.3]{SV} without an explicit proof, and it is a consequence of \cite[Prop.\,2.6]{Z} in the particular case in which $M$ is finitely generated and torsion; so, for the sake of completeness, we give here a detailed proof.  

\begin{proposition}\label{bernoulli}
Let $M$ be a $R$-module of finite valuation length and let 
\[
\beta\colon \bigoplus _{n \geq 1} M_n \to  \bigoplus _{n \geq 1} M_n
\] 
be the right Bernoulli shift, where $M_n=M$ for all $n$, i.e. $\left(\bigoplus _{n \geq 1} M_n\right)_\beta\cong M\otimes_{R}R[X]$. Then, 
$\widetilde{\ent}_v(\beta)=L_v(M)$. 
\end{proposition}
\begin{proof}
By Proposition \ref{torsionequal}, it is enough to check that $\ent_v(\beta)=L_v(M)$, since $M$ must be a torsion module, so also $ \bigoplus _{n \geq 1} M_n$ is torsion. From the equality 
\[
\ent_v(\beta,M_1)=\inf_n L_v\left(\frac{T_{n+1}(\beta, M_1 )}{T_n(\beta, M_1)}\right)
\]
 and since $\frac{T_{n+1}(\beta, M_1 )}{T_n(\beta, M_1)} \cong M$ for any $n$, we derive that 
$L_v(M)=\ent_v(\beta,M_1) \leq \ent_v(\beta)$. On the other hand, by \cite[Prop.\,4.2]{SV}, $\ent_v(\beta)= \sup_ F \ent_v(\beta,F)$, with $F$ finitely generated $R$-submodule of $ \bigoplus _{n \geq 1} M_n $. Given such an $F\leq  \bigoplus _{n \geq 1} M_n $, then $F \leq T_j(\beta, M_1)=T_j$ for a certain $j \geq 1$ so, by Corollary \ref{inf_coro}, we can conclude that
\begin{equation*}
\ent_v(\beta,F)\leq \ent_v(\beta, T_j)=\ent_v(\beta, M_1)=L_v(M).
\qedhere
\end{equation*}
%
\end{proof}

The next theorem is based on \cite[Thm.\,7.3]{SV}, that is, the Uniqueness Theorem for the valuation entropy $\ent_v$ on the category of torsion modules.

\begin{theorem}[Uniqueness Theorem]\label{UT} 
Let $R$ be the valuation domain of an archimedean non-discrete valuation $v:Q \to \R \cup  \{ \infty \}$ on its field of quotient $Q$, and let $L_v$ be the induced non-discrete length function on $\mod R$. The intrinsic valuation entropy $\widetilde{\ent}_v$ is the unique length function $L_X \colon \mod{R[X]} \to \R^*$ such that:
\begin{enumerate}
\item[\rm (1)] $L_X(M \otimes_R R[X]))=L_v(M)$ for all modules $M$ (of finite length);
\item[\rm (2)] $L_X(M_\phi)=v(s)$, for any linear transformation $\phi\colon M\to M$ of a finite dimensional $Q$-vector space, where $s$ is the leading coefficient of the characteristic polynomial of $\phi$ over $R$.
\end{enumerate}
\end{theorem}
\begin{proof} Since the $R[X]$-module $M \otimes_R R[X]$ is isomorphic to $ \bigoplus _{n \geq 1} M_n $ endowed with the right Bernoulli shift $\beta$, where $M_n=M$ for all $n$, Proposition \ref{bernoulli} shows that $\widetilde{\ent}_v$ satisfies condition (1) when $M$ has finite valuation length, and it is immediate to derive that (1) follows also when $L_v(M)= \infty$. Condition (2) is satisfied by $\widetilde{\ent}_v$ by Theorem \ref{IAYF}. 

Concerning  uniqueness, since a  length function is determined by the values it takes on cyclic modules, it is enough to prove that $L_X(T(\phi,x)_\phi)=\widetilde{\ent}_v(T(\phi,x)_\phi)$ for any $x\in M$, where $\phi\colon M \to M$ is an endomorphism of an $R$-module $M$. In view of the Addition Theorem, we can assume that $M$ is either torsion or torsion-free. In the first case \cite[Thm.\,7.3]{SV} shows that $L_X(T(\phi,x)_\phi)=\ent_v(T(\phi,x)_\phi)$, so the conclusion follows by Proposition \ref{torsionequal}. In the latter case, either
$T(\phi,x)_\phi \cong R[X]$, in case $T(\phi,x)$ is an $R$-module of infinite rank, or $T(\phi,x)_\phi\cong R[X]/(f(X))$, by Proposition \ref{cyclictrajectory}, where $f(X)$ is the characteristic polynomial of $\phi_{ \restriction T(\phi,x)}$ over $R$.

If $T(\phi,x)_\phi \cong R[X]$, then \cite[Thm.\,2]{NR} ensures that $L_X(T(\phi,x)_\phi)=\infty$, otherwise $L_X$ would be equivalent to $\rk_{R[X]}$, which is impossible. Also $\ient(T(\phi,x)_\phi)=\infty$, as a consequence of Proposition \ref{inf}.

If $T(\phi,x)_\phi \cong R[X]/(f(X))$, $L_X(T(\phi,x)_\phi)=v(s)$ by hypothesis, where $s$ is the leading coefficient of $f(X)$, and also $\widetilde{\ent}_v((T(\phi,x)_\phi))=v(s)$, by the IAYF. So we are done.
\end{proof}

\bigskip

\providecommand{\bysame}{\leavevmode\hbox to3em{\hrulefill}\thinspace}
\providecommand{\MR}{\relax\ifhmode\unskip\space\fi MR }
\providecommand{\MRhref}[2]{%
  \href{http://www.ams.org/mathscinet-getitem?mr=#1}{#2}
}
\providecommand{\href}[2]{#2}


\begin{thebibliography}{DBSV15}

\bibitem[AKM65]{AKM}
R.L. Adler, A.G. Konheim, M.H. McAndrew, \emph{Topological entropy},
  Trans. Amer. Math. Soc. \textbf{114} (1965), 309-319.

\bibitem[DBSV15]{DGSV}
D. Dikranjan, A. Giordano Bruno, L. Salce, S. Virili,
  \emph{Intrinsic algebraic entropy}, J. Pure Appl. Algebra
  \textbf{219} (2015), no. 7, 2933-2961.

\bibitem[DGB13]{DGB}
D. Dikranjan, A. Giordano Bruno, \emph{Discrete dynamical systems in
  group theory}, Proc. Conf. ``Advances in Group Theory and
  Applications 2011" Porto Cesareo, 2013.
  
  \bibitem[DGB16]{DG}
D. Dikranjan, A. Giordano Bruno, \emph{Entropy in Abelian groups}, Adv. Math. \textbf{208} (2016), 612-653.

\bibitem[DGSZ09]{DGSZ}
D. Dikranjan, B. Goldsmith, L. Salce, P. Zanardo,
  \emph{Algebraic entropy for abelian groups}, Trans. Amer. Math. Soc.
  \textbf{361} (2009), no. 7, 3401-3434.

\bibitem[FS85]{FS}
L. Fuchs, L. Salce, \emph{Modules over Valuation Domains},
  Lecture Notes Pure Appl. Math. vol. 97, M. Dekker, 1985.

\bibitem[FS01]{FS1}
L. Fuchs, L. Salce, \emph{Modules over Non-Noetherian Domains}, SURV no. 84, Amer.
  Math. Soc., 2001.


\bibitem[GBV12]{GV}
A. Giordano-Bruno, S. Virili, \emph{Algebraic Yuzvinski formula},
  J. Algebra \textbf{423} (2015), 114-147. 

\bibitem[GBV15]{GV2}
A. Giordano-Bruno, S. Virili,\emph{About the algebraic Yuzvinski formula}, Top. Algebra Appl.
  \textbf{3} (2015), no. 1, 86-103.
   
\bibitem[LW88]{LW} 
D. A. Lind, T. Ward, \emph{Automorphisms of solenoids and $p$-adic entropy}, Ergod. Th. \& Dynam. Sys. \textbf{8} (1988), 411--419.

\bibitem[NR65]{NR}
D.G. Northcott, M. Reufel, \emph{A generalization of the concept of
  length}, Quart. J. Math. Oxford Ser. (2) \textbf{16} (1965), 297-321.


\bibitem[SV16]{SV}
L. Salce, S. Virili, \emph{The Addition Theorem for algebraic entropies induced by
  non-discrete length functions}, Forum Math. \textbf{28} (2016), no. 6, 1143-1157.

\bibitem[SV17]{SV1}
L. Salce, S. Virili, \emph{Two new proofs concerning the intrinsic algebraic entropy},
  submitted.

\bibitem[SVV13]{SVV}
L. Salce, P. V{\'a}mos, S. Virili, \emph{{Length functions,
  multiplicities and algebraic entropy}}, Forum Math. \textbf{25} (2) (2013),
  255-282.

\bibitem[SZ09]{SZ}
L. Salce, P. Zanardo, \emph{A general notion of algebraic entropy and
  the rank-entropy}, Forum Math. \textbf{21} (2009), no. 4, 579-599.

\bibitem[Wei74]{W}
M.D. Weiss, \emph{Algebraic and other entropies of group endomorphisms},
Math. Syst. Theory \textbf{8} (1974), no. 3, 243-248.

\bibitem[{W}il15]{Willis_endo}
{G}.~{A}. {W}illis, \emph{The scale and tidy subgroups for endomorphisms of
  totally disconnected locally compact groups}, Math. {A}nn. \textbf{361}
  (2015), 403--442.

\bibitem[Zan11]{Z}
P. Zanardo, \emph{Multiplicative invariants and length functions over
  valuation domains}, J. Commut. Algebra \textbf{3} (2011), no. 4, 561-587.


\end{thebibliography}
\end{document}